\documentclass[11pt]{article}
\usepackage{color,amsmath, amsthm, amsfonts, amssymb, dsfont, graphicx, listings, graphics, epsfig, enumerate, bbm,dsfont}
\usepackage[pdftex]{hyperref}
\usepackage[margin=1in]{geometry}
\usepackage[fleqn,tbtags]{mathtools}
\usepackage[english]{babel}

\title{A characterization theorem for Aumann integrals}
\author{\c{C}a\u{g}{\i}n Ararat\thanks{Princeton University, Department of Operations Research and Financial Engineering, Princeton, NJ 08544, USA, cararat@princeton.edu.}
\and Birgit Rudloff\thanks{Princeton University, Department of Operations Research and Financial Engineering \& Bendheim Center for Finance, Princeton, NJ 08544, USA, tel: +1(609)258-4558, fax: +1(609)258-3791, brudloff@princeton.edu.}}
\date{November 18, 2014}

\addto\captionsenglish {}
\makeatletter \renewenvironment{proof}[1][\proofname] {\par\pushQED{\qed}\normalfont\topsep6\p@\@plus6\p@\relax\trivlist\item[\hskip\labelsep\bfseries#1\@addpunct{.}]\ignorespaces}{\popQED\endtrivlist\@endpefalse} \makeatother
\newtheorem{thm}{Theorem}[section]

\newtheorem{lem}[thm]{Lemma}
\newtheorem{prop}[thm]{Proposition}

\newtheorem{defn}[thm]{Definition}
\theoremstyle{definition}
\newtheorem{example}[thm]{Example}
\newtheorem{rem}[thm]{Remark}

\numberwithin{equation}{section}
\newcommand{\R}{\mathrm{I\negthinspace R}}
\newcommand{\sm}{\!\setminus\!}

\DeclareMathOperator{\cl}{cl}

\DeclareMathOperator{\co}{co}
\DeclareMathOperator{\interior}{int}

\let\abs=\envert

\newcommand{\G}{\mathcal{G}}
\newcommand{\F}{\mathcal{F}}
\newcommand{\X}{\mathcal{X}}
\newcommand{\A}{\mathcal{A}}
\renewcommand{\L}{\mathcal{L}}

\newcommand{\of}[1]{\ensuremath{\left( #1 \right)}}
\newcommand{\cb}[1]{\ensuremath{ \left\{ #1 \right\} }}

\newcommand{\ip}[1]{\ensuremath{ \left\langle #1 \right\rangle }}

\begin{document}
\maketitle
\thispagestyle{empty}
\begin{abstract}
A Daniell-Stone type characterization theorem for Aumann integrals of set-valued measurable functions will be proven. It is assumed that the values of these functions are closed convex upper sets, a structure that has been used in some recent developments in set-valued variational analysis and set optimization. It is shown that the Aumann integral of such a function is also a closed convex upper set. The main theorem characterizes the conditions under which a functional that maps from a certain collection of measurable set-valued functions into the set of all closed convex upper sets can be written as the Aumann integral with respect to some $\sigma$-finite measure. These conditions include the analog of the conlinearity and monotone convergence properties of the classical Daniell-Stone theorem for the Lebesgue integral, and three geometric properties that are peculiar to the set-valued case as they are redundant in the one-dimensional setting.
\\[.2cm]
\\[.2cm]
\\[.2cm]
The final publication in Set-Valued and Variational Analysis, DOI: 10.1007/s11228-014-0309-0, is available at Springer via \url{http://dx.doi.org/10.1007/s11228-014-0309-0}.								
\end{abstract}

\section{Introduction}
Integration of set-valued functions (multifunctions, correspondences) on a measurable space dates back to a paper by Aumann \cite{aumann} from 1965. Given a $\sigma$-finite measure, the Aumann integral of a measurable function whose values are subsets of $\R^{m}$, $m\in\{1,2,\ldots\}$, is defined as the set of all (vector) Lebesgue integrals of its integrable selections. The Aumann integral can be seen as a functional that maps a measurable set-valued function to a subset of $\R^{m}$. Conversely, one can consider a functional that maps from a certain class of measurable set-valued functions into a collection of subsets of $\R^{m}$ and look for conditions under which this functional can be written as the Aumann integral with respect to a measure. This type of result is well known for the Lebesgue integral, see \cite{stone,cinlar}, and is also referred to as the Daniell-Stone theorem. To the best of our knowledge, there seems to be no such characterizations of the Aumann integral in the literature so far. In the present paper, a special structure provided by the so-called upper sets is assumed for the values of the measurable set-valued functions and a Daniell-Stone type characterization theorem is proven for the Aumann integral.

Set relations have recently been of interest in set-valued variational analysis and set optimization. In \cite{andreasduality}, a reflexive and transitive relation is defined on the power set of $\R^{m}$ (more generally, a real linear space) with respect to a fixed convex cone $C\subseteq \R^{m}$. The set of equiva\-lence classes induced by the corresponding indifference relation of this set relation is given by $\mathcal{P}(\R^{m},C)=\cb{D\subseteq\R^{m}\mid D=D+C}$, whose elements are the so-called upper subsets of $\R^{m}$. It turns out that the set $\mathcal{P}(\R^{m},C)$ of all upper subsets has rather useful algebraic and order theoretic properties: On the algebraic side, an addition operation and a multiplication with positive scalars $\lambda\geq 0$ can be defined on $\mathcal{P}(\R^{m},C)$, which makes it a conlinear space, that is, a linear space except that it is closed under multiplication with positive scalars only. On the order theoretic side, the infimum and supremum of an arbitrary collection of upper sets are well-defined upper sets when $\mathcal{P}(\R^{m},C)$ is equipped with the usual superset relation $\supseteq$. Thus, $\mathcal{P}(\R^{m},C)$ is an order-complete lattice under $\supseteq$.

The properties of $\mathcal{P}(\R^{m},C)$ mentioned above and their consequences make it possible to generalize many concepts and results of variational analysis to functions whose values are upper sets. These include Legendre-Fenchel conjugation, directional derivatives, subd\-ifferentials, the Fenchel-Moreau biconjugation theorem, Lagrange and Fenchel-Rockafellar duality; see \cite{setoptsurvey} for a survey on set-valued convex analysis based on the so-called complete lattice approach. 
This approach also provides new points of view in vector optimization, see \cite{lohne,setoptsurvey}, and in financial mathematics, see, for instance,  \cite{conical} for the theory of set-valued risk measures based on upper sets. It turns out that many problems in vector or set optimization can be rewritten as problems for $\mathcal{P}(\R^{m},C)$-valued functions. In the present paper, we will thus focus on such functions. The complete lattice approach will play an essential role in obtaining the present results.

In classical Lebesgue integration, one version of the Daniell-Stone theorem states that a functional that maps into $[0, +\infty]$ from the set of all positive measurable functions on a measurable space can be written as the integral with respect to a measure if and only if it is conlinear (that is, linear except that multiplication with strictly negative scalars is not considered) and it preserves decreasing limits (monotone convergence property); see \cite[Theorem~I.4.21]{cinlar} for this version of the result and \cite{stone} for the original work of Stone. The aim of this paper is to prove an analogue of this theorem for the Aumann integrals of measurable functions whose values are closed and convex upper subsets of $\R^{m}$. The Aumann integrals of such functions are again closed convex upper sets as shown in Proposition~\ref{integralup}. This property makes it possible to use the same algebraic and order theoretic rules for the values of the functions and their integrals.

The main result of the present paper is Theorem~\ref{daniellstone}. It shows that a functional that maps measurable set-valued functions into closed convex upper sets can be written as the Aumann integral with respect to a measure if and only if the functional is conlinear, it preserves decreasing limits and it satisfies three additional properties that are peculiar to the set-valued framework. One of these properties ensures that the value of the functional at a given set-valued function can be computed in terms of the corresponding values at the supporting halfspaces of the function. This is already a well-known property of the Aumann integral as proven in \cite{survey}. The other two properties ensure that the functional maps halfspace-valued functions to halfspaces and ``point+cone"-valued functions to ``point+cone"s. 
It turns out that these three properties, which are redundant in the scalar case, suffice to complement the already existing properties of the scalar theory (conlinearity, monotone convergence and a technical condition to guarantee $\sigma$-finiteness) in order to obtain a Daniell-Stone type characterization for the Aumann integral.

\section{Closed convex upper sets}\label{uppersets}
Let $m$ be a strictly positive integer and consider the Euclidean space $\R^{m}$ with its usual topology and the usual inner product $\ip{\cdot, \cdot}$. Let $\interior D$, $\cl D$, $\co D$ denote the interior, closure, and the convex hull of a set $D\subseteq\R^{m}$, respectively. The \textit{Minkowski sum} of two sets $D, E\subseteq\R^{m}$ is defined as $D+E=\cb{d+e\mid d\in D,\;e\in E}$ with the convention $D+\emptyset=\emptyset+D=\emptyset$. Let $C\subseteq\R^{m}$ be a closed convex cone such that $0\in C$ and $C\neq \R^{m}$, which is to be fixed throughout. We also assume that $C$ has nonempty interior, that is, $\interior C\neq\emptyset$.

Define
\begin{equation*}
\G=\G(\R^{m},C)=\cb{D\subseteq\R^{m}\mid D=\cl\co(D+C)}.
\end{equation*}
An element of $\G$ is called an \textit{closed convex upper set}. This paper is concerned with the Aumann integration of functions on a measurable space whose values are closed convex upper sets. An immediate consequence of the Hahn-Banach theorem is that every $D\in\G$ can be written as
\begin{equation}\label{separation}
D=\bigcap_{w\in C^{+}\setminus\{0\}}\cb{z\in\R^{m}\mid \ip{z, w}\geq\inf_{y\in D}\ip{y, w}},
\end{equation}
where $C^{+}$ is the positive dual cone of $C$ given by
\begin{equation*}
C^{+}=\cb{w\in \R^{m}\mid \forall z\in C:\; \ip{z, w}\geq 0}.
\end{equation*}

Following \cite{andreasduality,setoptsurvey}, we summarize the algebraic and order theoretic properties of $\G$ that will be used in the present paper. Note that the Minkowski sum of two sets in $\G$ may fail to be closed. Instead, the addition operation $\oplus$ on $\G$ is defined by
\begin{equation*}
D\oplus E=\cl(D+E),\quad D, E\in\G.
\end{equation*}
For $D\in\G$ and $\lambda\in\R$, define $\lambda D=\{\lambda d\mid d\in D\}$ with the conventions $\lambda\emptyset=\emptyset$ for $\lambda\neq 0$ and $0D= C$. This is the usual Minkowski multiplication with scalars except for the second convention concerning the multiplication with zero. This convention guarantees that $\G$ is closed under multiplication with positive real numbers $\lambda\geq 0$. We will use shorthand notations such as $z-D=\cb{z}+(-1)D$ for $z\in\R^{m}$ and $D\in\G$. Finally, when the usual superset relation $\supseteq$ is used as a partial order, $\G$ is an order-complete lattice, where the infimum and supremum of a collection $\mathcal{D}\subseteq\G$ are provided by the formulae
\begin{equation*}
\inf\mathcal{D}=\cl\co\bigcup_{D\in\mathcal{D}}{D},\quad \sup\mathcal{D}=\bigcap_{D\in\mathcal{D}}{D},
\end{equation*}
respectively, with the conventions $\inf \emptyset=\emptyset$ and $\sup\emptyset=\R^{m}$; see \cite{andreasduality}.

\section{Measurable functions}
Let $(\X,\A)$ be a measurable space. Given a set-valued function $F\colon\X\rightarrow\G\sm\cb{\emptyset}$, the \textit{preimage} of a set $D\subseteq\R^{m}$ under $F$ is defined as
\begin{equation*}
F^{-1}(D)=\cb{x\in \X\mid F(x)\cap D\neq\emptyset}.
\end{equation*}

\begin{defn}\label{measurability}
A set-valued function $F\colon\X\rightarrow\G\sm\cb{\emptyset}$ is said to be \textit{measurable} if $F^{-1}(D)\in\A$ for every closed set $D\subseteq\R^{m}$. The set of all measurable set-valued functions $F\colon\X\rightarrow\G\sm\cb{\emptyset}$ is denoted by $\F$.
\end{defn}

This is the usual notion of measurability for set-valued functions as in the books \cite{rockafellar2,Molchanov,sdibook}. In Theorem~2.3 of \cite[Chapter 1]{Molchanov}, set-valued functions whose values are closed subsets of a Polish space are considered. Assuming the existence of a $\sigma$-finite measure $\nu$ under which $(\X,\A,\nu)$ is complete, it is shown that such a set-valued function is measurable in the sense of Definition~\ref{measurability} if and only if the preimage of every Borel set, or equivalently every open set, is measurable. Since we work with functions with values in $\G$, it is possible to provide another characterization of measurability in terms of simple sets of the form ``point minus cone'' as shown in the following proposition.

\begin{prop}\label{mainmeasurable}
Let $F\in\F$. The following are equivalent:
\begin{enumerate}[(i)]
\item $F$ is measurable.
\item $F^{-1}(D)\in\A$ for every $D\subseteq \R^{m}$ with $D=\cl\co(D-C)$.
\item $F^{-1}(y-C)\in\A$ for every $y\in \R^{m}$.
\end{enumerate}
\end{prop}

\begin{proof}
It is obvious that $(i)\Rightarrow (ii)\Rightarrow(iii)$. Note that the elements of $\G\sm\cb{\emptyset}$ are regular closed sets in the sense that $D=\cl\interior D$ for every $D\in\G$. Indeed, for $D\in\G\sm\{\emptyset\}$, note that $\interior D\neq\emptyset$ because for every $z\in D$ it holds $\emptyset\neq z+\interior C=\interior(z+C)\subseteq \interior (D+C)=\interior D$. Then, $D=\cl\interior D$ follows from \cite[Theorem~2.33]{rockafellar2}. Hence, by \cite[Example~14.7]{rockafellar2}, $F$ is measurable if and only if $\cb{x\in \X\mid y\in F(x)}\in\A$ for every $y\in\R^{m}$. Now fix $y\in\R^{m}$. We claim that
\begin{equation*}
F^{-1}(y-C)=\cb{x\in \X\mid y\in F(x)}.
\end{equation*}
Indeed, the $\supseteq$ part is clear since $0\in C$. For the $\subseteq$ part, let $x\in\X$ such that $F(x)\cap(y-C)\not=\emptyset$. Assume that $y\notin F(x)$. There exists $z\in C$ such that $y-z\in F(x)$. Thus, $y\in F(x)+z\subseteq F(x)+C=F(x)$, a contradiction. Thus, $y\in F(x)$. Hence, $(i)\Leftrightarrow(iii)$ follows.
\end{proof}

For set-valued functions $F, G\colon\X\rightarrow\G\sm\cb{\emptyset}$ and $\lambda\in\R$, the set-valued functions $F+G$, $F\oplus G$, $\lambda F$ are defined in the pointwise sense with the conventions of Section~\ref{uppersets}. These set-valued functions are measurable when $F, G$ are measurable; see \cite[Proposition~14.11]{rockafellar2}.

Denote by $\L^{0}(\R^{m})$ the linear space of all Borel measurable functions $f=\of{f_{1},\ldots,f_{m}}\colon\X\rightarrow\R^{m}$. For subsets of $\L^{0}(\R^{m})$, Minkowski sum and multiplication with scalars are defined analogously as in the case of the subsets of $\R^{m}$.

\begin{defn}
Let $F\in\F$. A function $f\in\L^{0}(\R^{m})$ is said to be a \textit{measurable selection} of $F$ if $f(x)\in F(x)$ for every $x\in\X$. The set of all measurable selections of $F$ is denoted by $\L^{0}(F)$.
\end{defn}

\begin{prop}
\cite[Corollary~14.6]{rockafellar2} Every $F\in\F$ has a measurable selection, that is, $\L^{0}(F)\neq\emptyset$.
\end{prop}

\section{The Aumann integral}
Let $\mu$ be a $\sigma$-finite measure on $(\X,\A)$ with $\mu(\X)>0$. The set of all functions $f\in\L^{0}(\R^{m})$ with $\int \abs{f_{i}}d\mu<+\infty$ for every $i\in\cb{1,\ldots,m}$ is denoted by $\L^{\mu}(\R^{m})$. For $f\in\L^{\mu}(\R^{m})$, define
\begin{equation*}
\int fd\mu=\of{\int f_{1}d\mu,\ldots,\int f_{m}d\mu}.
\end{equation*}

\begin{defn}\label{defAumann}
Let $F\in\F$. Denote by $\L^{\mu}(F)$ the set of all $\mu$-integrable measurable selections of $F$, that is, $\L^{\mu}(F)=\L^{\mu}(\R^{m})\cap\L^{0}(F)$. The \textit{(Aumann) integral} of $F$ with respect to $\mu$ is defined as
\begin{equation*}
\int Fd\mu=\cl\cb{\int fd\mu\mid f\in\L^{\mu}(F)}.
\end{equation*}
$F$ is said to be $\mu$-integrable if $\L^{\mu}(F)\neq\emptyset$.
\end{defn}

\begin{prop}\label{integralup}
The Aumann integral of a set-valued function $F\in\F$ is a closed convex upper set, that is, $\int Fd\mu\in\G$.
\end{prop}

The proof of Proposition~\ref{integralup} is given in Section~\ref{appendix1}.

\begin{prop}\label{selectionsum}
Let $F, G\in\F$ be $\mu$-integrable set-valued functions. Then,
\begin{equation*}
\int \of{F\oplus G}d\mu=\int Fd\mu\oplus\int Gd\mu.
\end{equation*}
\end{prop}

\begin{proof}
The result is provided by \cite[Theorem~5.4]{survey}.
\end{proof}

\begin{prop}\label{poshom-int}
For every $F\in\F$ and $\lambda\in\R_{+}$, it holds $\int \lambda F d\mu=\lambda\int Fd\mu$. In particular, if $F\equiv C$, then $\int Fd\mu = C$.
\end{prop}

\begin{proof}
The result is obvious when $\lambda>0$. The case $\lambda=0$ follows from the convention $0D=C$ for $D\in\G$ once we show that $\int C d\mu = C$. Indeed, it is clear that $0\in\L^{\mu}(C)$, hence $C = \cl(0+C) \subseteq  \int C d\mu \oplus C=\int C d\mu$ by Proposition~\ref{integralup}. Conversely, if $f\in\L^{\mu}(C)$, then, for every $w\in C^{+}\sm\{0\}$, we have $\ip{f(x), w}\geq 0$ for every $x\in\X$ and hence $\ip{\int fd\mu, w}\geq 0$. Therefore, $\int fd\mu \in C$. This proves $\int Cd\mu = C$.
\end{proof}

Given $A\in\A$, define the function $\mathbf{1}_{A} F\colon\X\rightarrow\G$ for $x\in\X$ by
\begin{equation}\label{indicatorwithfunction}
(\mathbf{1}_{A}F)(x)=\begin{cases}F(x)&\text{ if }x\in A\\  C&\text{ else}\end{cases}.
\end{equation}
Note that $\mathbf{1}_{A}F$ is not the same as the function $1_{A}F$ given by
\begin{equation*}
(1_{A}F)(x)=1_{A}(x) F(x)=\begin{cases}F(x)&\text{ if }x\in A\\\{0\}&\text{ else}\end{cases},
\end{equation*}
for $x\in\X$. Indeed, we have $\mathbf{1}_{A}F=1_{A}F+1_{A^{c}} C$. Since we work with functions whose values are closed convex upper sets, it is often more advantageous in our framework to work with $\mathbf{1}_{A}F$ rather than $1_{A}F$.

\begin{defn}\label{integraloverset}
Let $F\in\F$ and $A\in\A$. The integral of $F$ over the set $A$ with respect to $\mu$ is defined as
\begin{equation*}
\int_{A}F d\mu=\int \mathbf{1}_{A}F d\mu.
\end{equation*}
\end{defn}

\begin{prop}\label{indicatorintegral}
For every $F\in\F$ and $A\in\A$, it holds
\begin{equation*}
\int_{A}F d\mu=\cb{\int_{A}f d\mu\mid f\in \L^{\mu}(F)}\oplus C.
\end{equation*}
\end{prop}

The proof of Proposition~\ref{indicatorintegral} is given in Section~\ref{appendix2}.

\begin{rem}
When $\mu(A)>0$ in Proposition~\ref{indicatorintegral}, one has $\int_{A}Fd\mu=\cl\cb{\int_{A}fd\mu\mid f\in\L^{\mu}(F)}$, which coincides with the usual definition of (the closure of) the Aumann integral of $F$ over the set $A$ as in \cite{survey}. When $\mu(A)=0$, Proposition~\ref{indicatorintegral} gives $\int_{A}Fd\mu=C$ which deviates from the classical definition $\cl\cb{\int_{A}fd\mu\mid f\in\L^{\mu}(F)}=\cb{0}$. This is in total analogy to Definition~\ref{defAumann} and aligns with the algebraic structure on $\G$.
\end{rem}

Since closed convex sets are identified by their support functions, it is natural to consider the relationship between the Aumann integral and the (random) support function of $F\in\F$. Proposition~\ref{convexAumann} below shows that Aumann integration and computing support functions commute. It is a well-known result in Aumann integration and valid in a more general framework; see \cite[Theorem~5.4]{survey}.
In the context of the present paper, it leads to the representation \eqref{repr} of the Aumann integral as given below.

\begin{prop}\label{convexAumann}
Let $F\in\F$ be a $\mu$-integrable set-valued function. For every $w\in C^{+}\sm\{0\}$, it holds
\begin{equation*}
\inf_{f\in \L^{\mu}(F)}\int \ip{f, w}d\mu=\int\inf_{y\in F(x)} \ip{y, w}\mu(dx).
\end{equation*}
In particular,
\begin{equation}\label{repr}
\int Fd\mu=\bigcap_{w\in C^{+}\setminus\{0\}}\cb{z\in\R^{m}\mid \ip{z, w}\geq \int\inf_{y\in F(x)} \ip{y, w}\mu(dx)},
\end{equation}
\end{prop}

\begin{proof}
The first part of the result is provided by \cite[Theorem~5.4]{survey}. The second part follows directly from \eqref{separation} and the observation that $\inf_{z\in \int Fd\mu}\ip{z, w}=\inf_{f\in\L^{\mu}(F)}\int \ip{f, w}d\mu$ for $w\in C^{+}\sm\cb{0}$.
\end{proof}

Using Proposition~\ref{convexAumann}, we compute the integrals of two particular functions next. The results of these examples will be used in the proof of Theorem~\ref{daniellstone}, the main result of the paper. Let us denote the halfspace with normal vector $w\in C^{+}\sm\{0\}$ by \[H(w)=\cb{z\in\R^{m}\mid \ip{z, w}\geq 0}.\]

\begin{example}
\label{computeG}
Let $w\in C^{+}\sm\{0\}$ and $\xi\colon\X\rightarrow \R\cup\cb{-\infty}$ be a Borel measurable function with $\int \xi^{+}d\mu<+\infty$, where $\xi^{+}\coloneqq\max\cb{\xi, 0}$. Let $F(x)=\cb{z\in\R^{m}\mid \ip{z, w}\geq \xi(x)}$ for $x\in\X$. Hence, $F\in\F$ and, for every $\bar{w}\in C^{+}\sm\{0\}$ and $x\in\X$,
\begin{equation*}
\inf_{y\in F(x)}\ip{y, \bar{w}}=\begin{cases}k\xi(x) & \text{if }\bar{w}=kw\text{ for some }k\in(0,+\infty)\\-\infty & \text{ else}\end{cases}.
\end{equation*}
Since $\int \xi^{+}d\mu<+\infty$, it is easy to check that $f\coloneqq \frac{\xi^{+}}{\ip{c, w}}c\in\L^{\mu}(F)$, where $c\in \interior C$ is some fixed point. Hence, $F$ is $\mu$-integrable. Note that, for every $\bar{w}\in C^{+}\sm\cb{0}$,
\begin{equation*}
\int \inf_{y\in F(x)}\ip{y, \bar{w}}\mu(dx)=\begin{cases}k\int \xi d\mu & \text{if }\bar{w}=kw\text{ for some }k\in (0,+\infty)\\-\infty & \text{ else}\end{cases}.
\end{equation*}
Hence, by Proposition~\ref{convexAumann},
\begin{equation*}
\int Fd\mu = \bigcap_{k\in(0,+\infty)}\cb{z\in\R^{m}\mid \ip{z, kw}\geq k\int \xi d\mu} = \cb{z\in\R^{m} \mid \ip{z, w}\geq \int \xi d\mu}.
\end{equation*}
In particular, if $F\equiv H(w)$, that is, if $\xi\equiv 0$, then $\int Fd\mu = H(w)$.
\end{example}

\begin{example}\label{computeC}
Let $f\in\L^{\mu}(\R^{m})$ and $F(x)=f(x)+ C$ for $x\in\X$. Clearly $F\in\F$ is $\mu$-integrable. Note that, for every $w\in C^{+}\sm\{0\}$,
\begin{equation*}
\int \inf_{y\in F(x)}\ip{y, w}\mu(dx) = \int \of{\ip{f(x), w}+\inf_{y\in C}\ip{y, w}}\mu(dx)=\int \ip{f, w}d\mu=\ip{\int fd\mu, w}.
\end{equation*}
Hence, by Proposition~\ref{convexAumann},
\begin{equation*}
\int Fd\mu = \bigcap_{w\in C^{+}\setminus\{0\}}\cb{z\in\R^{m}\mid \ip{z, w}\geq \ip{\int fd\mu, w}} = \bigcap_{w\in C^{+}\setminus\{0\}}\of{\int fd\mu + H(w)} = \int fd\mu + C.
\end{equation*}
\end{example}

In Lebesgue integration, one version of the monotone convergence theorem states that the integral of the decreasing limit of measurable functions (dominated by some integrable function) is given as the decreasing limit of the corresponding integrals. In Proposition~\ref{contabove} below, a similar convergence result is proven for the Aumann integrals.

\begin{prop}\label{contabove}
Let $(F_{n})_{n\in\mathbb{N}}$ and $F$ be in $\F$. Suppose that $(F_{n})_{n\in\mathbb{N}}$ decreases to $F$ in the sense that, for $\mu$-a.e. $x\in\X$, $F_{n}(x)\subseteq F_{n+1}(x)$, $n\in\mathbb{N}$,  and $\cl\bigcup_{n\in\mathbb{N}} F_{n}(x)=F(x)$. Suppose that $F_{1}$ is $\mu$-integrable. Then, for every $n\in\mathbb{N}$, \[\int F_{n}d\mu\subseteq \int F_{n+1}d\mu,\] and \[\cl\bigcup_{n\in\mathbb{N}}\int F_{n}d\mu=\int Fd\mu.\]
\end{prop}

\begin{proof}
When $\mu$ is a probability measure, the result is provided by Theorem~1.44 of \cite[Chapter 2]{Molchanov}. The proof of that case relies on the classical monotone convergence theorem and works in the general $\sigma$-finite case as well.
\end{proof}

\section{The characterization theorem}

The aim of this section is to present a Daniell-Stone type characterization theorem for Aumann integrals of set-valued functions in $\F$, which is the set of all measurable set-valued functions $F\colon\X\rightarrow\G\sm\cb{\emptyset}$. Recall that $\G=\cb{D\subseteq \R^{m}\mid D=\cl\co\of{D+C}}$, where $C\neq\R^{m}$ is a closed convex cone with $\interior C\neq\emptyset$. Let $c\in\interior C$ be fixed.

\begin{thm}\label{daniellstone}
Let $\Phi\colon\F\rightarrow\G$ be a set-valued functional. Consider the following properties:
\begin{enumerate}
\item[\textup{\textbf{(A)}}] \textup{\textbf{Additivity:}} For every $F, G\in\F$ with $\Phi(F)\neq\emptyset$ and $\Phi(G)\neq\emptyset$, it holds $\Phi\of{F\oplus G}=\Phi(F)\oplus \Phi(G)$.
\item[\textup{\textbf{(P)}}] \textup{\textbf{Positive homogeneity:}} For every $F\in\F$ and $\lambda\in\R_{+}$, it holds $\Phi(\lambda F)=\lambda \Phi(F)$.
\item[\textup{\textbf{(C)}}] \textup{\textbf{Continuity from above:}} For every $(F_{n})_{n\in\mathbb{N}}$, $F$ in $\F$ with $\Phi(F_{1})\neq\emptyset$, $F_{n}(x)\subseteq F_{n+1}(x)$ for every $n\in\mathbb{N}$ and $x\in\X$, and $\cl\bigcup_{n\in\mathbb{N}} F_{n}(x)=F(x)$ for every $x\in \X$, it holds $\Phi(F_{n})\subseteq \Phi(F_{n+1})$ for every $n\in\mathbb{N}$, and $\cl\bigcup_{n\in\mathbb{N}} \Phi(F_{n})=\Phi(F)$.
\item[\textup{\textbf{(N)}}] \textup{\textbf{Nullity:}} If $F\equiv H(w)$ for some $w\in C^{+}\sm\cb{0}$, then $\Phi(F)=H(w)$.
\item[\textup{\textbf{(I) }}] \textup{\textbf{Indicator property:}} For every $\xi\in\L^{0}(\R_{+})$, either $\Phi(\xi c+C)=\emptyset$, or else there exists $k\in\R_{+}$ such that $\Phi(\xi c+C)=k c+C$. In addition, there exists at least one $\xi \in \L^{0}(\R_{++})$ such that $\Phi(\xi c+C)\notin\cb{\emptyset, C}$.
\item[\textup{\textbf{(S)}}] \textup{\textbf{Interchangeability with supporting halfspaces:}} Let $F\in\F$ with $\Phi(F)\neq\emptyset$. For $x\in\X$, $w\in C^{+}\sm\{0\}$, define the supporting halfspace of $F(x)$ with normal $w\in C^{+}\sm\{0\}$ by
\begin{equation}\label{supportinghalfspace}
F^{w}(x)=F(x)\oplus H(w)=\cb{z\in\R^{m}\mid \ip{z, w}\geq\inf_{y\in F(x)}\ip{y, w}}.
\end{equation}
Then, it holds $\Phi(F)=\bigcap_{w\in C^{+}\setminus\{0\}}\Phi(F^{w})$.
\end{enumerate}
$\Phi$ satisfies the above properties if and only if there exists a unique nonzero $\sigma$-finite measure $\mu$ on $(\X,\A)$ such that, for every $\mu$-integrable $F\in\F$,
\begin{equation*}
\Phi(F)=\int Fd\mu.
\end{equation*}
\end{thm}

The proof of Theorem~\ref{daniellstone} is given in Section~\ref{appendix3}.

Theorem~\ref{daniellstone} is a set-valued generalization of the following Daniell-Stone characterization of the Lebesgue integral: A functional on the set of all positive measurable functions that maps into $[0,+\infty]$ is the Lebesgue integral with respect to some measure if and only if it is additive and positively homogeneous, and it satisfies the monotone convergence property; see \cite[Theorem~I.4.21]{cinlar} for this version, and \cite{stone} for the original work. Indeed, these three properties are precisely equivalent to properties~(A), (P), (C), respectively, in the scalar case where $m=1$ and $C=\R_{+}$. Properties~(A) and (P) together are sometimes called ``conlinearity" meaning that the functional is linear except that negative scalars are not considered in multiplication.

Therefore, properties (N), (I) and (S) are the additional properties of the set-valued framework. Properties~(N) and (I) put regularity conditions on how the functional preserves the geometric properties of some special set-valued functions. In particular, Property~(N) assumes that a constant homogeneous halspace is mapped to itself under $\Phi$. Together with the aforementioned properties, it ensures that halfspace-valued functions are mapped to halfspaces under $\Phi$. Property~(I) provides a similar regularity condition when the functional is applied to functions having a special ``point plus cone" structure, which is already verified for Aumann integrals in Example~\ref{computeC}. In particular, Property~(I) describes the behavior of the functions of the form $\xi=1_{A}$ where $A\in\A$. Recalling \eqref{indicatorwithfunction}, note that $\xi c+C=1_{A} c+C=\mathbf{1}_{A}F$, where $F\equiv c+C$. The function $1_{A}c+C$ can be thought as the set-valued indicator function of the measurable set $A$. The second part of Property~(I) characterizes $\sigma$-finiteness (and the nontriviality) of the measure and also exists as an optional additional property in the scalar case; see \cite[Exercise~I.4.32]{cinlar}. Since the construction of the Aumann integral in Proposition~\ref{indicatorintegral} and Property~(S) in its relation to Proposition~\ref{convexAumann} already assume $\sigma$-finiteness of the underlying measure, this condition is needed for the characterization theorem in the set-valued case.

Property~(S) assumes that computing supporting halfspaces of functions on $\X$ commutes with the functional, which is verified for Aumann integrals in Proposition~\ref{convexAumann}. While Proposition~\ref{convexAumann} is stated in terms of the Lebesgue integrals of (scalar) support functions, this is not the case for Property~(S). Instead, we consider the supporting halfspace of a function as another function in $\F$ and state Property~(S) in terms of its Aumann integral.
In the scalar case where $m=1$ and $C=\R_{+}$, properties (N), (I) (the first part), and (S) become trivial and thus can be omitted.

\section{Proofs}
\subsection{Proof of Proposition~\ref{integralup}}\label{appendix1}
First, we show
\begin{equation}\label{uppersetequation}
\cb{\int f d\mu\mid f\in\L^{\mu}(F)}=\cb{\int f d\mu\mid f\in\L^{\mu}(F)}+C.
\end{equation}
The $\subseteq$ part is obvious since $0\in C$. For the $\supseteq$ part, let $f\in \L^{\mu}(F)$ and $z\in C$. Define $g\in\L^{0}(\R^{m})$ by $g(x)=f(x)+(\mu(A))^{-1}1_{A}(x) z$ for $x\in\X$, where $A\in\A$ is some nonempty set with $0<\mu(A)<+\infty$, and $1_{A}(x)=1$ for $x\in\A$ and $1_{A}(x)=0$ for $x\notin A$. Then, $g(x)\in F(x)+C=F(x)$ for every $x\in\X$ since $C$ is a cone with $0\in C$. Besides, $g\in\L^{\mu}(F)$ and $\int g d\mu=\int f d\mu+z$.

Note that taking closures in \eqref{uppersetequation} gives
\begin{equation*}
\int F d\mu=\cb{\int f d\mu\mid f\in\L^{\mu}(F)}\oplus C=\int Fd\mu\oplus C.
\end{equation*}

Next, we show that $\cb{\int fd\mu\mid f\in\L^{\mu}(F)}$ is a convex set. Indeed, let $f,g\in\L^{\mu}(F)$ and $\lambda\in(0,1)$. Clearly, $\lambda f+(1-\lambda)g\in\L^{\mu}(F)$ since the values of $F$ are convex sets. Since $\lambda\int fd\mu+(1-\lambda)\int gd\mu=\int(\lambda f+(1-\lambda)g)d\mu$, convexity follows. Hence, $\int Fd\mu$ is also convex. So
\begin{equation*}
\int Fd\mu=\co\of{\int Fd\mu\oplus C}=\cl\co\of{\int Fd\mu+C},
\end{equation*}
that is, $\int Fd\mu\in\G$.

\subsection{Proof of Proposition~\ref{indicatorintegral}}\label{appendix2}
By the definition of the integral,
\begin{equation*}
\begin{split}
\int_{A}F d\mu&=\cl\cb{\int g d\mu\mid g\in \L^{\mu}(\mathbf{1}_{A}F)}\\
&=\cl\of{\cb{\int_{A}g^{1} d\mu\mid g^{1}\in\L^{\mu}(\mathbf{1}_{A}F)}+\cb{\int_{A^{c}}g^{2} d\mu\mid g^{2}\in\L^{\mu}(\mathbf{1}_{A}F)}}.
\end{split}
\end{equation*}
Here, the $\subseteq$ part of the second equality is obvious and the $\supseteq$ part follows from the observation that if $g^{1},g^{2}\in\L^{\mu}(\mathbf{1}_{A}F)$, then $g=1_{A}g^{1}+1_{A^{c}}g^{2}\in \L^{\mu}(\mathbf{1}_{A}F)$ by triangle inequality and $\int gd\mu=\int_{A} g^{1}d\mu+\int_{A^{c}}g^{2}d\mu$.

We claim that
\begin{equation}\label{claim1}
\cb{\int_{A}g^{1} d\mu\mid g^{1}\in \L^{\mu}(\mathbf{1}_{A}F)}=\cb{\int_{A}f d\mu\mid f\in \L^{\mu}(F)}.
\end{equation}
To show $\subseteq$, let $g^{1}\in\L^{\mu}(\mathbf{1}_{A}F)$. For some arbitrarily fixed $g^{0}\in \L^{\mu}(F)$, we have $f=1_{A}g^{1}+1_{A^{c}}g^{0}\in\L^{\mu}(F)$ by triangle inequality and $\int_{A}g^{1} d\mu=\int_{A}f d\mu$ as well. The $\supseteq$ part is trivial since $f\in\L^{\mu}(F)$ implies $1_{A}f\in\L^{\mu}(\mathbf{1}_{A}F)$ (as $0\in  C$) and we have $\int_{A}1_{A}f d\mu=\int_{A}f d\mu$. Thus, \eqref{claim1} holds.

If $\mu(A^{c})=0$, then the result follows immediately since
\begin{equation*}
\begin{split}
\int_{A}F d\mu&=\cl\of{\cb{\int_{A}f d\mu\mid f\in \L^{\mu}(F)}+\cb{0}}=\cl\cb{\int_{A}f d\mu\mid f\in \L^{\mu}(F)}\\
&=\cb{\int_{A}f d\mu\mid f\in \L^{\mu}(F)}\oplus C,
\end{split}
\end{equation*}
where the first equality uses \eqref{claim1} and the last equality is due to the fact that $\int_{A}Fd\mu\in\G$; see Proposition~\ref{integralup} and Definition~\ref{integraloverset}.

Finally, suppose that $\mu(A^{c})>0$. To finish the proof, we claim that
\begin{equation}\label{claim2}
\cb{\int_{A^{c}}g^{2} d\mu\mid g^{2}\in \L^{\mu}(\mathbf{1}_{A}F)}=C.
\end{equation}
To show $\subseteq$, let $g^{2}\in \L^{\mu}(\mathbf{1}_{A}F)$. For every $x\in A^{c}$, we have $g^{2}(x)\in C$. Since $C$ is a closed convex cone, this is equivalent to the following: For every $x\in A^{c}$ and $w\in C^{+}\sm\{0\}$, it holds $\ip{g^{2}(x), w}\geq 0$.  Thus,
\begin{equation*}
\ip{\int_{A^{c}}g^{2} d\mu, w}=\int_{A^{c}}\ip{g^{2}(x), w}\mu(dx)\geq \int_{A^{c}}0d\mu=0,
\end{equation*}
for every $w\in C^{+}\sm\{0\}$, which shows that $\int_{A^{c}}g^{2} d\mu\in C$. To show $\supseteq$ in \eqref{claim2}, let $z\in C$. Since $\mu$ is $\sigma$-finite and $\mu(A^{c})>0$, there exists $A_{0}\in\A$ such that $0<\mu(A^{c}\cap A_{0})\leq\mu(A_{0})<+\infty$. Note that for every $f\in\L^{\mu}(F)$, we have $g^{2}=1_{A}f+1_{A^{c}\cap A_{0}} z\in\L^{\mu}(\mathbf{1}_{A}F)$ since
\begin{equation*}
\int\abs{g^{2}}d\mu\leq \int_{A}\abs{f}d\mu+\mu(A^{c}\cap A_{0})\abs{z}<+\infty,
\end{equation*}
by triangle inequality, where $\abs{\cdot}$ is some arbitrary fixed norm on $\R^{m}$. Besides, for such $g^{2}$,
\begin{equation*}
\int_{A^{c}}g^{2} d\mu=\mu(A^{c}\cap A_{0}) z.
\end{equation*}
Hence,
\begin{equation*}
\cb{\int_{A^{c}}g^{2} d\mu\mid g^{2}\in\L^{\mu}(\mathbf{1}_{A}F)}\supseteq \cb{\mu(A^{c}\cap A_{0}) z\mid z\in  C}=\mu(A^{c}\cap A_{0})C= C
\end{equation*}
since $C$ is a cone and $\mu(A^{c}\cap A_{0})>0$. Thus, \eqref{claim2} holds.

\subsection{Proof of Theorem~\ref{daniellstone}}\label{appendix3}

Before we prove Theorem~\ref{daniellstone}, we begin with a technical result to be used in the proof.

\begin{lem}\label{generating}
For every $f\in \L^{\mu}(\R^{m})$, there exists some $\xi\in\L^{\mu}(\R_{+})$ such that $\xi c-f\in\L^{\mu}(C)$.
\end{lem}

\begin{proof}
Let
\begin{equation*}
\xi\coloneqq \of{\sup_{w\in C^{+}\sm\cb{0}}\frac{\ip{f, w}}{\ip{c, w}}}^{+}.
\end{equation*}
Recall that $\ip{c, w}>0$ for every $w\in C^{+}\sm\cb{0}$. Hence, $\ip{\xi c -f, w}=\xi \ip{c, w}-\ip{f, w}\geq 0$ for every $w\in C^{+}\sm\cb{0}$, that is, $\xi c-f\in \L^{0}(C)$. It remains to show that $\xi\in\L^{\mu}(\R_{+})$. Noting that $\frac{\ip{f, w}}{\ip{c, w}}=\frac{\ip{f, kw}}{\ip{c, kw}}$ for every $k>0$ and $w\in C^{+}\sm\cb{0}$, it follows that
\begin{equation*}
\xi= \of{\sup_{w\in D(c)}\ip{f, w}}^{+},
\end{equation*}
where $D(c)\coloneqq \cb{w\in C^{+}\colon \ip{c, w}=1}$. Since $\interior C\neq \emptyset$, the dual cone $C^{+}$ is pointed, that is, $C^{+}\cap - C^{+}=\cb{0}$. Hence, an elementary exercise in convex analysis yields that $D(c)$ is a bounded set. Therefore, letting
\begin{equation*}
a\coloneqq \sup_{w\in D(c)}\max_{i\in\cb{1,\ldots,m}}\abs{w_{i}}<+\infty,
\end{equation*}
it holds
\begin{equation*}
\xi \leq \sup_{w\in D(c)}\sum_{i=1}^{m}\abs{f_{i}}\abs{w_{i}}\leq a\sum_{i=1}^{m}\abs{f_{i}}.
\end{equation*}
Since $f\in\L^{\mu}(\R^{m})$, it follows that $\xi\in\L^{\mu}(\R_{+})$.
\end{proof}

\begin{proof}[Proof of Theorem~\ref{daniellstone}.]
If there exists a unique nonzero $\sigma$-finite measure $\mu$ on $(\X,\A)$ such that $\Phi(F)=\int Fd\mu$ for every $\mu$-integrable $F\in\F$, then all properties are already satisfied as shown in  Proposition~\ref{selectionsum} (Property~(A)), Proposition~\ref{poshom-int} (Property~(P)), Proposition~\ref{contabove} (Property~(C)), Example~\ref{computeG} (Property~(N)), Example~\ref{computeC} (Property~(I)), Proposition~\ref{convexAumann} (Property~(S)). 

Suppose that $\Phi$ satisfies the properties listed in the theorem.

Let $\xi\in\L^{0}(\R_{+})$. If $\Phi(\xi c+C)=\emptyset$, then set $\varphi(\xi)=+\infty$. Otherwise, there exists $k\in\R_{+}$ with $\Phi(\xi c+C)=kc+C$ and set $\varphi(\xi)=k$. Here, such $k$ is unique because for $\bar{k}\in\R$,
\begin{equation}\label{uniquenessargument}
 kc+ C=\bar{k}c+C\quad\Longrightarrow\quad k=\bar{k}.
\end{equation}
This is an easy exercise which follows from the fact that $c\in\interior C$. Hence $\varphi(\xi)$ is well-defined. Next, we show that the function $\varphi\colon\L^{0}(\R_{+})\rightarrow\R_{+}\cup\{+\infty\}$ satisfies the assumptions of the classical Daniell-Stone theorem: 

\begin{enumerate}[(i)]
\item Monotonicity: Let $\xi, \zeta\in\L^{0}(\R_{+})$ with $\xi(x)\leq\zeta(x)$ for every $x\in\X$. We claim $\varphi(\xi)\leq\varphi(\zeta)$. If $\varphi(\zeta)=+\infty$, the claim holds trivially. Assume that $\varphi(\zeta)<+\infty$, that is, $\Phi(\zeta c+C)\neq\emptyset$. Then, $F^{1}\coloneqq \zeta c+C$, $F\coloneqq F^{n}\coloneqq \xi c + C$ for $n\in\cb{2,3,\ldots}$ satisfy the assumptions of Property~(C). Hence, $\emptyset\neq\Phi(\zeta c + C)\subseteq \Phi(\xi c + C)$. By Property~(I), $\emptyset\neq\varphi(\zeta)c+C\subseteq\varphi(\xi)c+C$. By \eqref{uniquenessargument}, it follows that $\varphi(\xi)\leq \varphi(\zeta)$.

\item Additivity: Let $\xi, \zeta\in\L^{0}(\R_{+})$. We claim $\varphi(\xi+\zeta)=\varphi(\xi)+\varphi(\zeta)$. By the monotonicity of $\varphi$, it holds $\max\cb{\varphi(\xi),\varphi(\zeta)}\leq\varphi(\xi+\zeta)$. Hence, if $\varphi(\xi)=+\infty$ or $\varphi(\zeta)=+\infty$, then the claim holds trivially. Assume that $\varphi(\xi)<+\infty$ and $\varphi(\zeta)<+\infty$, that is, $\Phi(\xi c+C)\neq\emptyset$ and $\Phi(\zeta c+C)\neq\emptyset$. Hence, by Property~(A),
\begin{equation*}
\Phi((\xi+\zeta)c+C)=\Phi\of{(\xi c +C)\oplus (\zeta c+C)}=\Phi(\xi c + C)\oplus\Phi(\zeta c + C) = (\varphi(\xi)+\varphi(\zeta))c+C\neq\emptyset.
\end{equation*}
By Property~(I) and \eqref{uniquenessargument}, it follows that $\varphi(\xi+\zeta)=\varphi(\xi)+\varphi(\zeta)$.

\item Positive homogeneity: Let $\xi\in\L^{0}(\R_{+})$ and $\lambda\in\R_{+}$. We claim $\varphi(\lambda\xi)=\lambda\varphi(\xi)$. If $\lambda=0$, then $\varphi(0)c+C=\Phi(C)=\Phi(0 C)=0 \Phi(C)= C$ by Property~(P). Hence, the claim holds with $\varphi(\lambda\xi)=\varphi(0)=0$. Assume that $\lambda>0$. By Property~(P) and the fact that $C$ is a cone, it holds $\Phi(\lambda\xi c+C)=\Phi(\lambda(\xi c+C))=\lambda\Phi(\xi c+C)$. If $\varphi(\xi)=+\infty$, that is, $\Phi(\xi c + C)=\emptyset$, then $\Phi(\lambda\xi c+C)=\emptyset$, that is, $\varphi(\lambda\xi)=+\infty$. If $\varphi(\xi)<+\infty$, then $\varphi(\lambda\xi)c+C=\lambda(\varphi(\xi)c+C)=\lambda\varphi(\xi)c+C$. By \eqref{uniquenessargument}, it follows that $\varphi(\lambda\xi)=\lambda\varphi(\xi)$.

\item Continuity from above: Let $(\xi^{n})_{n\in\mathbb{N}}, \xi^{\infty}$ be in $\L^{0}(\R_{+})$ such that $\varphi(\xi^{1})<+\infty$, $\xi^{n}(x)\geq\xi^{n+1}(x)$ for every $n\in\mathbb{N}$ and $x\in\X$, and $\lim_{n\rightarrow\infty}\xi^{n}(x)=\xi^{\infty}(x)$ for every $x\in\X$. We claim $\varphi(\xi^{n})\geq \varphi(\xi^{n+1})$ for every $n\in\mathbb{N}$ and $\lim_{n\rightarrow\infty}\varphi(\xi^{n})=\varphi(\xi^{\infty})$. The first part of the claim follows by the monotonicity of $\varphi$. For the second part, let $F_{n}\coloneqq \xi^{n}c+C$ for every $n\in\mathbb{N}\cup\cb{+\infty}$. Then, $\Phi(F_{1})\neq\emptyset$ and hence,
\begin{equation*}
\Phi(F_{\infty})=\cl \bigcup_{n\in\mathbb{N}}\Phi(F_{n})=\cl\bigcup_{n\in\mathbb{N}}(\varphi(\xi^{n})c+C)=\of{\lim_{n\rightarrow\infty}\varphi(\xi^{n})}c+C\neq\emptyset
\end{equation*}
by Property~(C). By \eqref{uniquenessargument}, it follows that $\varphi(\xi^{\infty})=\lim_{n\rightarrow\infty}\varphi(\xi^{n})$.

\end{enumerate}
Hence, by the classical Daniell-Stone theorem, there exists a unique measure $\mu$ such that
\begin{equation}\label{scalardaniellstone}
\varphi(\xi)=\int\xi d\mu
\end{equation}
for every $\xi\in\L^{0}(\R_{+})$. The second part of Property~(I) implies that $\mu$ is nonzero and $\sigma$-finite; see \cite[Exercise~I.4.32]{cinlar}.

Finally, we prove that $\Phi(F)=\int Fd\mu$ for every $\mu$-integrable function $F\in\F$. We proceed in three main steps.

\noindent \textbf{Step 1: } Let $F\in\F$ be a \textit{negative set-valued function} in the sense that $0\in F(x)$, that is, $C\subseteq F(x)$, for every $x\in \X$. Clearly, $F$ is $\mu$-integrable as $0\in\L^{\mu}(F)$. We prove $\Phi(F)=\int Fd\mu$. We first prove this claim for the case of halfspace-valued functions and then approximate $F$ by its supporting halfspaces.
\begin{enumerate}
\item[\textbf{(1a)}] If $F\equiv H(w)$ for some $w\in C^{+}\setminus\{0\}$, then $\Phi(F)=H(w)=\int Fd\mu$ by Property~(N) and Example~\ref{computeG}.

\item[\textbf{(1b)}] Suppose that $F=\{z\in\R^{m}\mid \ip{z, w}\geq -\xi\}$ for some Borel function $\xi\in\L^{0}(\R_{+})$ with $\int \xi d\mu<+\infty$ and some $w\in C^{+}\sm\cb{0}$. Then,
\begin{equation}\label{halfspaceequation}
F=\cb{z\in\R^{m}\mid \ip{z, w}\geq -\frac{\xi}{\ip{c, w}}\ip{c, w}}=\of{-\frac{\xi}{\ip{c, w}}c+C}\oplus H(w).
\end{equation}
Since $\int\xi d\mu<+\infty$, the first part of Property~(A) implies that
\begin{equation*}
\begin{split}
C&=\Phi(C)=\Phi\of{-\frac{\xi}{\ip{c, w}}c+C}\oplus\Phi\of{\frac{\xi}{\ip{c, w}}c+C}\\
&=\Phi\of{-\frac{\xi}{\ip{c, w}}c+C}\oplus \of{\varphi\of{ \frac{\xi}{\ip{c, w}}}c+C} =\varphi\of{ \frac{\xi}{\ip{c, w}}}c+\Phi\of{-\frac{\xi}{\ip{c, w}}c+C}
\end{split}
\end{equation*}
so that
\begin{equation*}
\Phi\of{-\frac{\xi}{\ip{c, w}}c+C}=-\varphi\of{\frac{\xi}{\ip{c, w}}}c+C=\of{-\int \frac{\xi}{\ip{c, w}}d\mu}c+C.
\end{equation*}
Hence, by \eqref{halfspaceequation}, the first part of Property~(A), and Example~\ref{computeG}, it follows that
\begin{equation*}
\begin{split}
\Phi(F)&=\of{\of{-\int\frac{\xi}{\ip{c, w}}d\mu}c+C}\oplus H(w)=\of{-\int\frac{\xi}{\ip{c, w}}d\mu}c+H(w)\\
&=\cb{z\in\R^{m}\mid \ip{z, w}\geq -\int \xi d\mu}=\int Fd\mu.
\end{split}
\end{equation*}

\item[\textbf{(1c)}] Suppose that $F=\{z\in\R^{m}\mid \ip{z, w}\geq -\xi\}$ for some Borel function $\xi\colon\X\rightarrow\R_{+}\cup\{+\infty\}$ and $w\in C^{+}\sm\cb{0}$. 
Then, there exists a sequence $(\xi^{n})_{n\in\mathbb{N}}$ in $\L^{0}(\R_{+})$ with $\xi^{n}\leq \xi^{n+1}$ and $\int \xi^{n}d\mu<+\infty$ for every $n\in\mathbb{N}$, and $\lim_{n\rightarrow\infty}\xi^{n}(x)=\xi(x)$ for every $x\in\X$. Define $F_{n}\coloneqq \{z\in\R^{m}\mid \ip{z, w}\geq -\xi^{n}\}$ for $n\in\mathbb{N}$. By the previous case, Property~(C), monotone convergence theorem for Lebesgue integrals, and Example~\ref{computeG}, it holds
\begin{equation*}
\begin{split}
\Phi(F)&=\cl\bigcup_{n\in\mathbb{N}}\Phi(F_{n})=\cl\bigcup_{n\in\mathbb{N}} \cb{z\in\R^{m}\mid \ip{z, w}\geq -\int \xi^{n} d\mu}\\
&= \cb{z\in\R^{m}\mid \ip{z, w}\geq -\lim_{n\rightarrow\infty}\int \xi^{n} d\mu}=\cb{z\in\R^{m}\mid \ip{z, w}\geq -\int \xi d\mu}=\int Fd\mu.
\end{split}
\end{equation*}

\item[\textbf{(1d)}] Let $F\in\F$ be an arbitrary negative set-valued function. For $w\in C^{+}\sm\{0\}$, define the supporting halfspace $F^{w}$ by \eqref{supportinghalfspace}. Using Property~(S), Step~(1c), Proposition~\ref{convexAumann} for the respective equalities, we have
\begin{equation*}
\Phi(F)=\bigcap_{w\in C^{+}\setminus\{0\}}\Phi(F^{w})=\bigcap_{w\in C^{+}\setminus\{0\}}\int F^{w}d\mu=\int Fd\mu.
\end{equation*}
\end{enumerate}

\noindent \textbf{Step 2:} Let $F=f+C$ for some $f\in\L^{\mu}(\R^{m})$. We show that $\Phi (F)= \int Fd\mu = \int fd\mu + C$.
\begin{enumerate}
\item[\textbf{(2a)}] If $f\in -\L^{\mu}(C)$, then Step~1 and Example~\ref{computeC} imply that $\Phi(f+C)=\int \of{f+C}d\mu = \int fd\mu +C$.
\item[\textbf{(2b)}] If $f=\xi c$ for some $\xi\in \L^{\mu}(\R_{+})$, then $\varphi(\xi)=\int \xi d\mu <+\infty$ by \eqref{scalardaniellstone}. Hence, Property~(I) implies that $\Phi(\xi c+C)=\of{\int \xi d\mu}  c+ C$. By Example~\ref{computeC}, $\Phi(f+C)=\int\of{f+C}d\mu$.

\item[\textbf{(2c)}] If $f\in \L^{\mu}(\R^{m})$, we may write $f=\xi c-g$ for some $\xi\in\L^{\mu}(\R_{+})$ and $g\in \L^{\mu}(C)$ by Lemma~\ref{generating}. The previous two cases and the first part of Property~(A) yield that
\begin{equation*}
\Phi(f+C)=\Phi\of{(\xi c+C)\oplus(-g+C)}=\of{\int (\xi c) d\mu + C}\oplus \of{\int (-g)d\mu + C}=\int fd\mu + C.
\end{equation*}
By Example~\ref{computeC}, $\Phi(f+C)=\int\of{f+C}d\mu$.
\end{enumerate}

\noindent \textbf{Step 3: } Let $F\in\F$ be $\mu$-integrable. Fix $f\in\L^{\mu}(F)$. Then $G\coloneqq -f+F$ is a negative set-valued function as defined in Step~1. By the previous two steps and the first part of Property~(A),
\begin{equation*}
\begin{split}
-\int fd\mu + \Phi(F)&=\of{-\int fd\mu +C}\oplus\Phi(F)=\Phi(-f+C)\oplus\Phi(F)=\Phi(G)\\
&=\int Gd\mu = \int (-f+C)d\mu\oplus\int Fd\mu=-\int fd\mu+\int Fd\mu.
\end{split}
\end{equation*}
It follows that $\Phi(F)=\int Fd\mu$.

The uniqueness of $\mu$ follows from its definition and \eqref{uniquenessargument}.
\end{proof}

\begin{rem}
In the proof of Theorem~\ref{daniellstone}, the measure $\mu$ can alternatively be constructed without reference to the classical Daniell-Stone theorem. Indeed, for each $A\in\A$, one can take $\xi=1_{A}$ in Property~(I) and define $\mu(A)\coloneqq\varphi(\xi)$. Hence, $\Phi(1_{A}c+C)=\mu(A)c+C$ if $\mu(A)<+\infty$ and $\Phi(1_{A}c+C)=\emptyset$ otherwise. Then, analogous to the properties (i)-(iv) of $\varphi$ in the proof of Theorem~\ref{daniellstone}, the following properties of the set function $\mu$ can be checked:
\begin{enumerate}
\item $\mu$ is monotone, that is, $A\subseteq B$ implies $\mu(A)\leq\mu(B)$ for every $A, B\in\A$.
\item $\mu$ is additive, that is, $\mu(A+B)=\mu(A)+\mu(B)$ for every disjoint $A, B\in\A$.
\item $\mu(\emptyset)=0$.
\item $\mu$ is continuous from above, that is, if $(A_{n})_{n\in\mathbb{N}}$ is a sequence in $\A$ with $\mu(A_{1})<+\infty$ and $A_{n}\supseteq A_{n+1}$ for every $n\in\mathbb{N}$, then $\mu(\bigcap_{n\in\mathbb{N}}A_{n})=\lim_{n\rightarrow\infty}\mu(A_{n})$.
\end{enumerate}
Hence, $\mu$ is a measure.
\end{rem}

\section*{Acknowledgment}
The second author acknowledges that her research was supported by the National Science Foundation under award DMS-1007938.

\end{document}